\documentclass[11pt]{article}
\usepackage{geometry}
\usepackage{color}
\usepackage{float}
\usepackage{textcomp}
\usepackage{url}
\usepackage{amsthm}
\usepackage{amsmath}
\usepackage{amssymb}%
\usepackage{mdframed}
\usepackage{multirow}
\usepackage{changes}
\usepackage{hyperref}
\usepackage{authblk} 

\usepackage{mathtools}
\usepackage{booktabs}    
\usepackage{caption} 
\usepackage{multirow}
\usepackage{multicol}    
\usepackage{array}       
\usepackage{graphicx}
\graphicspath{{figuras/}}
\usepackage{empheq,amsmath}

\usepackage
[ 
lined, 
boxruled, 
commentsnumbered
]
{algorithm2e}
\DecMargin{0.1cm} 

\usepackage{cleveref}

\newtheorem{theorem}{Theorem}[section]
\newtheorem{lemma}[theorem]{Lemma}
\newtheorem{corollary}[theorem]{Corollary}

\newtheorem{remark}[theorem]{Remark}
\newtheorem{definition}[theorem]{Definition}

\newcommand{\N}{\mathbb{N}}
\newcommand{\R}{\mathbb{R}}

\newcommand{\bi}{\begin{itemize}}
\newcommand{\ei}{\end{itemize}}
\newcommand{\ba}{\begin{array}}
\newcommand{\ea}{\end{array}}

\usepackage{amsthm}
\usepackage{subcaption}

\newcommand{\M}{\mathcal{M}}
\newcommand{\E}{\mathcal{E}}
\def \Ext {\mathrm{Ext}}
\def \Int {\mathrm{Int}}
\newcommand{\tangent}[1]{T_{#1} \mathcal{M}}
\newcommand{\inner}[3]{\langle #2, #3 \rangle_{#1}}
\newcommand{\norm}[2]{\left\| #2 \right\|_{#1}}
\newcommand{\grad}{\operatorname{grad}}
\newcommand{\proj}[1]{\operatorname{Proj}_{\tangent{#1}}}

\newtheoremstyle{customAssumption}
{}
{}
{}
{1.25ex}
{}
{}
{1.5ex}
{\textbf{A\thmnumber{#2}.}}

\theoremstyle{customAssumption}
\newtheorem{assumption}{A\hskip-1ex}

\begin{document}

\title{\textbf{Riemannian optimization with finite-difference gradient approximations}}


\author[1]{Timoth\'e Taminiau\thanks{Email: timothe.taminiau@uclouvain.be. Supported by ``Fonds spéciaux de Recherche", UCLouvain.}}

\author[1]{Estelle Massart\thanks{Email: estelle.massart@uclouvain.be.}}

\author[1]{Geovani Nunes Grapiglia\thanks{Email: geovani.grapiglia@uclouvain.be. Partially supported by FRS-FNRS, Belgium (Grant CDR J.0081.23).}}

\affil[1]{Université Catholique de Louvain,  ICTEAM/INMA, Avenue Georges Lemaître 4, B-1348, Louvain-la-Neuve, Belgium}

\date{\today}

\maketitle

\begin{abstract}

Derivative-free Riemannian optimization (DFRO) aims to minimize an objective function using only function evaluations, under the constraint that the decision variables lie on a Riemannian manifold. The rapid increase in problem dimensions over the years calls for computationally cheap DFRO algorithms, that is, algorithms requiring as few function evaluations and retractions as possible. We propose a novel DFRO method based on finite-difference gradient approximations that relies on an adaptive selection of the finite-difference accuracy and stepsize that is novel even in the Euclidean setting. When endowed with an intrinsic finite-difference scheme, that measures variations of the objective in tangent directions using retractions, our proposed method requires $O(d\epsilon^{-2})$ function evaluations and retractions to find an $\epsilon$-critical point, where $d$ is the manifold dimension. We then propose a variant of our method when the search space is a Riemannian submanifold of an $n$-dimensional Euclidean space. This variant relies on an extrinsic finite-difference scheme, approximating the Riemannian gradient directly in the embedding space, assuming that the objective function can be evaluated outside of the manifold. This approach leads to worst-case complexity bounds of $O(d\epsilon^{-2})$ function evaluations and $O(\epsilon^{-2})$ retractions. We also present numerical results showing that the proposed methods achieve superior performance over existing derivative-free methods on various problems in both Euclidean and Riemannian settings.



\end{abstract}

\section{Introduction} \label{sec:1}

In this work, we consider the problem of minimizing a continuously differentiable function over a $d$-dimensional Riemannian manifold $\mathcal{M}$. Riemannian optimization, which has expanded considerably in the last two decades, is the standard approach for addressing this problem \cite{Absil2008,Boumal2023,Sato2021}. In this work, we assume that the derivatives of the objective function are not available, preventing the use of gradient-based Riemannian optimization methods. The lack of derivative information typically arises when the objective involves a black-box component, as in the design of adversarial attacks for neural network classifiers. These attacks seek an input perturbation that leads to a failure of the classifier, without having access to its architecture, and are naturally formulated as a derivative-free \emph{Euclidean} optimization problems \cite{Li2023,Ughi2021}. The advent of geometric deep learning and the design of neural networks with manifold-valued inputs \cite{Chakraborty2022,Huang2016} calls for derivative-free \emph{Riemannian} optimization (DFRO) algorithms for adversarial attacks design and robustness quantification. Other applications of DFRO arise, e.g., in the control of robotic systems \cite{Jacquier2020}, and in blind source separation \cite{Selvan2013}.

While in the Euclidean setting, derivative-free optimization is a rather mature topic (\cite{Audet2017,Conn2009a,Larson2019}), considerably fewer works address DFRO. Several works extended direct search methods to Riemannian manifolds \cite{Cavarretta2025,Dreisigmeyer07,Dreisigmeyer2018,Fong2022,Kungurtsev2023,Selvan2013}. These methods search for a new candidate by evaluating the objective on a mesh of trial points around the current iterate. When the objective function is sufficiently smooth, an alternative strategy extends classical derivative-based optimization algorithms by replacing derivatives by their finite-difference approximations. The local linear structure of the manifold is then used to define local sets of coordinates, allowing for the generalization of finite-differences schemes to Riemannian manifolds. 
Since this approximation is based on local variations along curves on the manifold,
we refer to it as an \emph{intrinsic finite-difference gradient}.
It has been used in several DFRO algorithms \cite{He2024,Louzeiro2024,Maass2022}
and generalized to simplex gradients, in which the reference directions form a
spanning set of the tangent space at the reference point \cite{Najafi2025}. Other works addressed Gaussian smoothing, which estimates the gradient from a two-point function evaluation scheme \cite{Li2023}, and finite-difference Hessian approximations  \cite{Boumal2015}.

As the problem dimension increases, ensuring the computational efficiency of DFRO algorithms becomes a major challenge. Their computational cost is governed by their total number of function evaluations and retractions. Computationally expensive function evaluations are common in derivative-free optimization; they may for example require solving a set of partial differential equations \cite{Karbasian2022}. As for retractions, note that, while the design of numerically cheap retractions has been the focus of dedicated research \cite{Absil2012,Absil2014,Oviedo2023,Sato2018}, their computational cost remains typically high for several manifolds of interest, which motivated the design of retraction-free Riemannian optimization methods, see, e.g., \cite{Ablin2022,Ablin2024,Gao2022,Goyens2025,Liu2024,Vary2024}. As a whole, it is therefore crucial to derive DFRO algorithms with a worst-case complexity in terms of function evaluations and retractions that is as low as possible. 

Unfortunately, the above-cited finite-difference DFRO methods are either not equipped with complexity results, or are equipped with results derived under stronger problem assumptions, such as the knowledge of some smoothness constants of the objective \cite{He2024,Li2023}, an assumption that is too strong for many problems of interest.  In derivative-free \emph{Euclidean} optimization, a strategy to avoid the use of smoothness constants is to rely on adaptive finite-difference accuracy and stepsize parameters \cite{Grapiglia2022,Grapiglia2023}. Under the assumption that $f(\,\cdot\,)$ is lower bounded and has Lipschitz continuous gradient, these methods need at most $O(d\epsilon^{-2})$ function evaluations, with $d$ the dimension of the search space, to find an $\epsilon$-critical point, i.e. a point $\bar{x}$ such that $\|\nabla f(\bar{x})\| \leq \epsilon$, with no information on the smoothness constants of the objective. Note that, for direct-search methods, a quadratic complexity in problem dimension was shown to be unavoidable \cite{Dodangeh2015}. To our knowledge, the design of DFRO algorithms with adaptive parameter selection (i.e., that do not relying on the knowledge of parameter constants), with associated complexity analysis, is still an open question.

The contributions of this paper are twofold. First, assuming that we have no access to the smoothness constants of the objective, we propose a finite-difference DFRO algorithm that relies on an adaptive selection scheme for the stepsize and finite-difference accuracy, which is novel even in the Euclidean setting. This scheme is motivated by the observation that the method proposed in \cite{Grapiglia2023} (in the Euclidean setting) tunes the stepsize and finite-difference accuracy based on one single parameter that approximates the Lipschitz constant of the gradient of the objective. However, the level of accuracy required in this estimation differs depending on the goal considered: while more conservative estimates should be used in the finite-difference computations, the use of sharper estimates in the stepsize selection may result in larger steps and speed up convergence. Inspired by \cite{Davar2025}, we propose a novel DFRO algorithm, based on intrinsic finite-difference gradients, that relies on two adaptive estimates of the Lipschitz constant of the gradient of the objective, used respectively in the finite-difference accuracy and stepsize computation. We prove that our algorithm finds an  $\epsilon$-critical point (i.e., a point at which the norm of the Riemannian gradient is no greater than $\epsilon$) after at most $O(d\epsilon^{-2})$ function evaluations and $O(d\epsilon^{-2})$ retractions, and show numerically an improvement over the method proposed in \cite{Grapiglia2023} in the Euclidean setting.

As a second contribution, we decrease the retraction complexity of our DFRO algorithm by proposing a variant that relies on an alternative \emph{extrinsic finite-difference scheme}, for optimization over Riemannian submanifolds of a Euclidean space. This scheme approximates the gradient from local variations along tangent directions \emph{in the embedding space}, allowing the saving of $d$ retractions per finite-difference gradient estimation. We prove that the worst-case complexities of this variant of our DFRO algorithm are $\mathcal{O}(d \epsilon^{-2})$ functions evaluations and $\mathcal{O}(\epsilon^{-2})$ retractions, with the latter being independent of the problem dimension. Numerical results are provided that show a substantial computational cost reduction compared to the intrinsic scheme. 

 The structure of the paper is as follows. In \Cref{sec:prelim}, we provide a short reminder of key concepts in Riemannian optimization, specify the problem assumptions and give preliminary lemmas that we will need in the convergence analysis of our methods.  We present in \Cref{sec:mainDRFOmethod} our DFRO method based on the intrinsic finite-difference gradients, with associated worst-case complexity bounds, and in \Cref{sec:extrinsic_scheme} our variant that relies on an extrinsic finite-difference scheme for optimization on Riemannian submanifolds of a Euclidean space. Finally, \Cref{sec:numerics} contains numerical experiments that compare our proposed methods with state-of-the-art methods regarding both function evaluations and running time on benchmark problems.

\paragraph{Notation:} We write $\tangent{x}$ the tangent space to the manifold $\M$ at $x \in \M$, $\inner{x}{\cdot}{\cdot}$ the Riemannian metric at $x$, and $\| \cdot \|_x$ the associated Riemannian norm. The tangent bundle of $\M$ is written $T\M$. Let $R : T\mathcal M \to \mathcal M$ be a retraction on $\mathcal M$.
For each $x \in \mathcal M$, the mapping $R_x : T_x \mathcal M \to \mathcal M$
is defined by $R_x(\eta) := R(x,\eta)$. The Riemannian gradient of $f : \M \rightarrow \R$ is written $\grad f(x) \in \tangent{x}$. When the manifold is a Riemannian submanifold of a Euclidean space $\E$ with respect to the canonical metric $\langle \,\cdot\,,\, \cdot\, \rangle$, we denote $\proj{x} : \E \rightarrow \tangent{x}$ the orthogonal projection of any vector $\eta \in \E$ onto $\tangent{x}$.

\section{Preliminary material} \label{sec:prelim}

We start this section with a short reminder of the key concepts from Riemannian optimization used in this work, and describe next our problem assumptions and a couple of lemmas required for the convergence analyses. 

\subsection{Basic tools and definitions in Riemannian Optimization}
A Riemannian manifold is a (smooth) manifold whose tangent space (i.e., local linear approximation) at any $x \in \M$ (written $T_x \M$) is endowed with an inner product $\langle \cdot, \cdot \rangle_x$ that varies smoothly with $x$, see, e.g., \cite{Absil2008,Boumal2023}. This Riemannian metric allows to define the Riemannian gradient of $f$ at $x$, as the unique vector $\grad f(x) \in T_x \M$ that satisfies $Df(x)[\eta] = \langle \grad f(x),\eta \rangle_x$ for all $\eta \in T_x \M$, where $Df(x)[\eta]$ is the directional derivative of $f$ in the direction $\eta$. When the Riemannian gradient is not available, it can be approximated by the intrinsic finite-difference gradient 
\begin{equation}  \label{eq:intrinsic-finite-diff}
g_h(x) = \sum_{l = 1}^d \frac{f(R_x(h e_l(x))) - f(x)}{h} e_l(x),
\end{equation}
where $\left\{e_{1}(x),\ldots,e_{d}(x)\right\}$ is a basis of $T_{x}\mathcal{M}$, $h>0$ controls the finite-difference accuracy, and $R_{x}(\,\cdot\,)$ is defined by a retraction. Retractions allow to make a step on the manifold from an arbitrary point $x \in \M$ along an arbitrary tangent vector $\eta \in T_x \M$. More precisely, a retraction is a smooth mapping 
$R : T\mathcal{M} \to \mathcal{M}$, with $(x,\eta) \mapsto R_x(\eta)$,
where $T\mathcal{M}$ is the tangent bundle of $\mathcal{M}$, 
satisfying $R_x(0) = x$ and $D R_x(0)[\cdot]$ is the identity operator (the so-called local rigidity condition), see, e.g.,  \cite[§3.5.3 and Def. 4.1.1]{Absil2008}. In this work, we assume that $\M$ is endowed with a retraction $R$ such that, for all $x \in \M$, $R_x : T_x \M \to \M$ is defined over the entire $T_x \M$. This assumption holds for several important examples of retraction, it fails for example for some projection-based retractions \cite{Absil2012} where the projection is only defined on a neighborhood of $0_x \in T_x \M$; we refer, e.g., to \cite{Boumal2019} for a complexity analysis of classical Riemannian optimization algorithms that accounts for locally defined retractions, by restricting the stepsize to the domain of definition of the retraction at each iteration. While similar ideas could be used in this work, we relied on our above-mentioned assumption on the retraction to simplify the presentation of our methods and their convergence analyses.

\subsection{Problem assumptions and auxiliary results}
We consider the optimization problem 
\begin{equation}  \label{eq:P}  \tag{P}
\min_{x \in \M} f(x),
\end{equation}
where we make the two following assumptions on the objective function $f$. The first one is a generalization to manifolds of the classical Lipschitz-smoothness assumption (see \cite{Boumal2019}), while the second simply requires the objective to be lower-bounded.

\begin{assumption} \label{assumption1}
$f(\,\cdot\,)$ is $(L_\M, R)$-smooth, that is for all $x \in \M$,
\begin{equation*}
|f(R_x(\eta))-f(x) - \inner{x}{\grad f(x)}{\eta}| \leq \frac{L_\M}{2} \norm{x}{\eta}^2, \quad \forall \eta \in \tangent{x}.
\end{equation*}
\end{assumption}

\begin{assumption} \label{assumption2}
There exists $f_\text{low} \in \R$ such that $f(x) \geq f_\text{low}$, for all $x \in \M$.
\end{assumption}

Since we will consider different finite-difference gradient schemes in \Cref{sec:mainDRFOmethod} and \Cref{sec:extrinsic_scheme}, we encompass here these two schemes in a broader notion of approximate Riemannian gradient, namely, tangent vectors parametrized by some parameter $h$ allowing to control their accuracy with respect to the exact Riemannian gradient. 

\begin{definition}  \label{def:approximate_gradient}
An approximate Riemannian gradient of a function $f : \mathcal{M} \to \mathbb{R}$ is a mapping $g : [0,+\infty) \times \mathcal{M} \to T\mathcal{M}$, $ (h,x) \mapsto g_h(x)$, such that
\begin{equation}
\norm{x}{\grad f(x) - g_h(x)} \leq C_f h,
\end{equation}
for some constant $C_f > 0$ that does not depend on $x$.
\end{definition}

Thereafter, we present two auxiliary lemmas that will be key in the convergence analysis of the methods proposed in this paper. The first lemma shows that the relative error of the approximate Riemannian gradient is bounded by some constant as soon as  $h$ is small enough and we have not yet reached $\epsilon$-criticality. 

\begin{lemma}[Adapted from Lemma 2 in \cite{Grapiglia2023}] \label{lemma2}
Given $\epsilon > 0$ and $x \in \M$ such that $\norm{x}{\grad f(x)} > \epsilon$, let $g_h(x)$ be an approximate Riemannian gradient at $x$, satisfying \Cref{def:approximate_gradient} for some constant $C_f>0$. If $h \leq \frac{\epsilon}{5C_f}$, then
\begin{equation*}
\norm{x}{g_h(x)} \geq \frac{4\epsilon}{5},
\end{equation*}
and
\begin{equation*}
\norm{x}{\grad f(x) - g_h(x)} \leq \frac{1}{4} \norm{x}{g_h(x)}.
\end{equation*}
\end{lemma}

\begin{proof}
\Cref{def:approximate_gradient} with $h \leq \frac{\epsilon}{5C_f}$ implies that
\begin{equation*}
\norm{x}{\grad f(x)-g_h(x)} \leq C_f h \leq \frac{\epsilon}{5}.
\end{equation*}
By the triangle inequality,
\begin{equation*}
\epsilon < \norm{x}{\grad f(x)} \leq \norm{x}{\grad f(x)-g_h(x)} + \norm{x}{g_h(x)} \leq \frac{\epsilon}{5} + \norm{x}{g_h(x)},
\end{equation*}
hence
\begin{equation*}
\norm{x}{g_h(x)} \geq \frac{4\epsilon}{5}
\end{equation*}
and
\begin{equation*}
\norm{x}{\grad f(x)-g_h(x)} \leq \frac{\epsilon}{5} \leq \frac{1}{4}\norm{x}{g_h(x)}.
\end{equation*}
\end{proof}


The next lemma uses this bound on the relative error of the gradient approximation to provide a derivative-free Armijo-like decrease of the objective under Lipschitz-smoothness of the objective (i.e., \Cref{assumption1}).


\begin{lemma}[Adapted from Lemma 1 in \cite{Grapiglia2023}] \label{lemma3}
Suppose that \Cref{assumption1} holds. Let $x \in \M$ and $g_h(x) \in \tangent{x}$ be such that
\begin{equation} \label{eq:relative_error}
\norm{x}{\grad f(x) - g_h(x)} \leq \frac{1}{4} \norm{x}{g_h(x)}, 
\end{equation}
and let  $x^+ = R_x \left( -\frac{1}{\sigma} g_h(x) \right)$ for $\sigma \geq L_\M$, then
\begin{equation*}
f(x) - f(x^+) \geq \frac{1}{4\sigma} \norm{x}{g_h(x)}^2.
\end{equation*}
\end{lemma}

\begin{proof}
By \Cref{assumption1}, there holds
\begin{align}
f(x^+) &= f\left(R_x\left(-\frac{1}{\sigma} g_h(x)\right)\right) \\
&\leq f(x) - \frac{1}{\sigma} \inner{x}{\grad f(x)}{g_h(x)} + \frac{L_\M}{2 \sigma^2} \norm{x}{g_h(x)}^2 \notag \\
&= f(x) + \frac{1}{\sigma} \inner{x}{g_h(x)-\grad f(x)}{g_h(x)} + \frac{
L_\M-2\sigma}{2 \sigma^2} \norm{x}{g_h(x)}^2 \notag \\
&\leq f(x) + \frac{1}{\sigma}\norm{x}{g_h(x)-\grad f(x)} \norm{x}{g_h(x)} + \frac{L_\M-2\sigma}{2 \sigma^2} \norm{x}{g_h(x)}^2. \label{eq:lipschitz_bound}
\end{align}
Combining this with \eqref{eq:relative_error} gives
\begin{align*}
f(x) - f(x^+) &\geq \left( \frac{2\sigma-L_\M}{2\sigma^2} - \frac{1}{4 \sigma} \right) \norm{x}{g_h(x)}^2 = \frac{3\sigma -2L_\M}{4\sigma^2} \norm{x}{g_h(x)}^2 \geq \frac{1}{4\sigma} \norm{x}{g_h(x)}^2.
\end{align*}
\end{proof}

\section{A novel finite-difference DFRO method} \label{sec:mainDRFOmethod}

We are now ready to present our main DFRO method, which is novel even in the Euclidean setting. The main specificity of this method is an adaptive scheme for the finite-difference accuracy and stepsize, that exploits information gathered on the objective over past iterations. In particular, our method does not require knowing the  smoothness constant $L_\M$ in A\ref{assumption1}, in contrast with existing DFRO methods that use smoothness constants in the stepsize computation \cite{He2024,Li2023}.

\begin{algorithm}[h]
\caption{Intrinsic Riemannian finite-difference method (Int-RFD)}
\label{alg:intrinsicRFD}
\textbf{Step 0.} Let $x_0 \in \M$, $\sigma_0 > 0$, $\tau_0 \geq \sigma_0$, $\epsilon > 0$,  set $k:=0$. Let $d$ be the dimension of $\M$. \\
\textbf{Step 1.1.} Let $h_k = \frac{2\epsilon}{5\sqrt{d}\tau_k}$, compute the intrinsic finite-difference gradient 
\begin{equation} \label{eq:step11}
g_{h_k}(x_k) = \sum_{l=1}^d \frac{f(R_{x_k}(h_k e_l(x_k))) - f(x_k)}{h_k} e_l(x_k),
\end{equation}
where $\{e_1(x_k),\dots,e_d(x_k)\}$ is an orthonormal basis of $\tangent{x_k}$. \\
\textbf{Step 1.2.} If 
\begin{equation}  \label{eq:step12}
\norm{x_k}{g_{h_k}(x_k)} < \frac{4\epsilon}{5},
\end{equation}
set $x_{k+1} = x_k$, $\sigma_{k+1} = \sigma_k$, $\tau_{k+1} = 2\tau_k$ and $k:=k+1$, go back to Step 1.1. \\
\textbf{Step 1.3.} If 
\begin{equation} \label{eq:step13}
f(x_k) - f\left(R_{x_k} \left(-\frac{1}{\sigma_k} g_{h_k}(x_k)\right)\right) \geq \frac{1}{4 \sigma_k} \norm{x_k}{g_{h_k}(x_k)}^2,
\end{equation}
set $x_{k+1} = R_{x_k} \left(-\frac{1}{\sigma_k} g_{h_k}(x_k)\right)$, $\sigma_{k+1} = \frac{\sigma_k}{2}$, $\tau_{k+1} = \tau_k$ and $k:=k+1$, go back to Step 1.1. \\
\textbf{Step 1.4} Set $x_{k+1} = x_k$ and $\sigma_{k+1} = 2 \sigma_k$. If $\sigma_{k+1} > \tau_k$ set $\tau_{k+1} = 2 \tau_k$ and $k:=k+1$, go back to Step 1.1. Otherwise set $\tau_{k+1} = \tau_k$ and $k:=k+1$, go back to Step 1.3. \\[0.2cm]
\end{algorithm}

Our intrinsic Riemannian finite-difference method, relying on the intrinsic finite-difference scheme \eqref{eq:intrinsic-finite-diff}, is presented in \Cref{alg:intrinsicRFD}.  The computation of a finite-difference approximation to the Riemannian gradient is done at Step 1.1, and the search for a stepsize that ensures a derivative-free Armijo-type sufficient decrease condition is in Step 1.3. The approximate gradient accuracy $h_k$ and stepsize used in these two steps are determined by two parameters $\tau_k$ and $\sigma_k$ that estimate the constant $L_\M$ in A\ref{assumption1} during optimization. These two parameters (whose update strategy is described in the next paragraph) differ in their roles: the first estimate $\tau_k$ is conservative and ensures sufficiently accurate finite-difference approximations whereas the second estimate $\sigma_k$ is an optimistic estimate used for the stepsize. When it is strictly smaller than $\tau_k$, it allows the method to perform larger steps. The use of two distinct estimates is the key modification with respect to \cite{Grapiglia2023} and was inspired by a similar idea proposed for trust-region derivative-free methods \cite{Davar2025}. 

The two sequences of parameters $\sigma_k$ and $\tau_k$ are updated as follows. When the norm of the estimated gradient is too small at Step 1.2, the parameter $\tau_k$ is increased in order to refine the accuracy of the gradient approximation. Indeed, by \Cref{lemma2}, \eqref{eq:step12} either implies that $x_k$ is an $\epsilon$-critical point of $f$ (in which case the algorithm succeeded), or that the finite-difference gradient has insufficient accuracy to ensure a guaranteed decrease of the objective by \Cref{lemma3},  we then increase $\tau_k$. Conversely, when \eqref{eq:step12} does not hold, we enter Step 1.3, where the sufficient decrease condition \eqref{eq:step13} is checked for a stepsize $1/\sigma_k$. If the decrease condition is satisfied, we accept the step and decrease $\sigma_k$ to allow for larger steps in future iterations. If not, the stepsize is decreased (by increasing $\sigma_k$) and the condition is checked again with this smaller stepsize. Note that, according to \Cref{lemma3} and provided $\tau_k$ is large enough, \eqref{eq:step13} will be satisfied as soon as $\sigma_k \geq L_\M$. If the updated value $\sigma_{k+1}$ becomes larger than $\tau_k,$ 
the failure of the decrease condition is likely due to an inaccurate finite-difference gradient which prevents \Cref{lemma3} to hold, and we increase $\tau_k$ and recompute an approximate gradient of higher accuracy in Step 1.1. 



We next provide a theoretical analysis of the convergence of \Cref{alg:intrinsicRFD}. For this, we organize the iterations in \Cref{alg:intrinsicRFD} in four categories, depending on whether they provide a decrease of the objective or not, and if they do not, the reason of failure. We thus define four disjoint sets $\mathcal{U}^{(1)}$, $\mathcal{U}^{(2)}$, $\mathcal{U}^{(3)}$ and $\mathcal{S}$ as follows.
\begin{itemize}
\item \textbf{Unsuccessful iterations of type I} ($k \in \mathcal{U}^{(1)}$): these iterations do not provide an objective decrease, and the failure comes from an approximate gradient $g_{h_k}(x_k)$ whose norm is too small, satisfying \eqref{eq:step12}. As discussed above, in this case we increase $\tau_k$.

\item \textbf{Unsuccessful iterations of type II} ($k \in \mathcal{U}^{(2)}$): the norm of the approximate gradient is large enough to enter Step 1.3, but the sufficient decrease \eqref{eq:step13} does not hold. We then increase $\sigma_k$, and get $\sigma_{k+1} > \tau_k$. We thus suspect a lack of accuracy of the finite-difference gradient explaining the failure of the decrease condition, and increase $\tau_k$.

\item \textbf{Unsuccessful iterations of type III} ($k \in \mathcal{U}^{(3)}$): In Step 1.3, the sufficient decrease \eqref{eq:step13} does not hold. We then increase $\sigma_k$ and get $\sigma_{k+1} \leq \tau_k$. In this case, we keep $\tau_k$ constant and avoid the computation of a new approximate gradient.


\item \textbf{Successful iterations} ($k \in \mathcal{S}$):  
The sufficient decrease condition \eqref{eq:step13} in Step 1.3 is satisfied and the step is accepted. The stepsize parameter $\sigma_k$ is decreased.
\end{itemize}

Let us write the number of iterations to reach $\epsilon$-criticality as
\begin{equation*}
T(\epsilon) = \inf \{ k \in \N \, | \, \norm{x_k}{\grad f(x_k)} \leq \epsilon \}.
\end{equation*}
The subsets of successful and unsuccessful iterations before reaching $\epsilon$-criticality are thus:
\[ \mathcal{S}_{T(\epsilon)} = \mathcal{S} \cap \{0,1, \dots, T(\epsilon)-1\}, \quad  \text{and} \quad \mathcal{U}_{T(\epsilon)}^{(i)} = \mathcal{U}^{(i)} \cap \{0,1, \dots, T(\epsilon)-1\}, \quad i = 1,2,3.\]

We are now ready to derive a worst-case complexity analysis of \Cref{alg:intrinsicRFD} in terms of function evaluations and retractions. It turns out that \Cref{alg:intrinsicRFD} is guaranteed to converge to an $\epsilon$-critical point, assuming that $\tau_k$ is bounded from above, 
as stated by the following result.


\begin{theorem} \label{theorem0}
Suppose that \Cref{assumption2} holds. Let us assume that there exists $\tau_{\max} > 0$ such that 
\begin{equation*}
\tau_k \leq \tau_{\max}, \quad \forall k = 0, 1, \dots, T(\epsilon)-1,
\end{equation*}
 then the number of function and retraction evaluations performed by \Cref{alg:intrinsicRFD} to reach an $\epsilon$-critical point of $f$ (i.e., a point $x^*\in \M$ such that $\|\grad f(x^*)\|_{x^*} \leq \epsilon$), written $FE(\epsilon)$ and $RE(\epsilon)$, satisfy
\begin{align}
\text{FE}(\epsilon) &\leq (FE_{g}+2)\left[2\log_{2}\left(\frac{\tau_{\max}}{\tau_{0}}\right)+\frac{25}{2}\tau_{\max}(f(x_{0})-f_{\text{low}})\epsilon^{-2}\right], \label{eq:bound1}\\
\text{RE}(\epsilon) &\leq (RE_{g}+1)\left[2\log_{2}\left(\frac{\tau_{\max}}{\tau_{0}}\right)+\frac{25}{2}\tau_{\max}(f(x_{0})-f_{\text{low}})\epsilon^{-2}\right],
\label{eq:bound2}
\end{align}
where $\text{FE}_g$ and $\text{RE}_g$ are respectively the number of new function evaluations and retractions to compute one finite-difference gradient $g_h(x)$.
\end{theorem}

\begin{proof}
Let $(x_k)_{k=0}^{T(\epsilon)-1}$ be a sequence of iterates generated by \Cref{alg:intrinsicRFD}. We start by finding upper bounds on the cardinality of the sets $\mathcal{S}_{T(\epsilon)}$, $\mathcal{U}_{T(\epsilon)}^{(1)}$, $\mathcal{U}_{T(\epsilon)}^{(2)}$ and $\mathcal{U}_{T(\epsilon)}^{(3)}$. Note that, by \Cref{assumption2},
\[ f(x_0) - f_\text{low} \geq \sum_{k=0}^{T(\epsilon)-1} (f(x_k) - f(x_{k+1})) = \sum_{k \in \mathcal{S}_{T(\epsilon)}} (f(x_k) - f(x_{k+1})). \]
Moreover, each successful iteration $k \in \mathcal{S}_{T(\epsilon)}$ satisfies
\begin{equation*}
f(x_k) - f(x_{k+1}) \geq \frac{1}{4 \sigma_k} \norm{x_k}{g_{h_k}(x_k)}^2 \geq \frac{1}{4\sigma_k} \left( \frac{4\epsilon}{5} \right)^2. 
\end{equation*}
Since $\sigma_{k}\leq\tau_{k}$ for all $k$, it follows from our assumption that
\begin{equation*}
f(x_0) - f_\text{low}  \geq \frac{4 \epsilon^2}{25 \tau_{\max}} |S_{T(\epsilon)}|
\end{equation*}
or equivalently, 
\begin{equation*}
|\mathcal{S}_{T(\epsilon)}| \leq \frac{25}{4} \tau_{\max} (f(x_0) - f_\text{low}) \epsilon^{-2}.
\end{equation*}
On the other hand, note that the update rules of $\tau_k$ and $\sigma_k$ imply
\[\tau_{T(\epsilon)-1} = 2^{|\mathcal{U}_{T(\epsilon)}^{(1)}| + |\mathcal{U}_{T(\epsilon)}^{(2)}|} \tau_0 \leq \tau_{\max} \qquad \text{and} \qquad \sigma_{T(\epsilon)-1} \leq 2^{|\mathcal{U}_{T(\epsilon)}^{(2)}|+|\mathcal{U}_{T(\epsilon)}^{(3)}|-|\mathcal{S}_{T(\epsilon)}|} \tau_0  \leq \tau_{\max}. \]
We deduce that 
\[|\mathcal{U}_{T(\epsilon)}^{(1)}| + |\mathcal{U}_{T(\epsilon)}^{(2)}| \leq \log_2\left( \frac{\tau_{\max}}{\tau_0} \right) \quad \text{and} \quad |\mathcal{U}_{T(\epsilon)}^{(2)}| + |\mathcal{U}_{T(\epsilon)}^{(3)}| \leq |\mathcal{S}_{T(\epsilon)}| + \log_2\left(\frac{\tau_{\max}}{\tau_0}\right). \]
Thus,
\begin{eqnarray}
T(\epsilon)&=&|\mathcal{U}_{T(\epsilon)}^{(1)}| + |\mathcal{U}_{T(\epsilon)}^{(2)}|+|\mathcal{U}_{T(\epsilon)}^{(3)}|+|\mathcal{S}_{T(\epsilon)}|\nonumber\\
&\leq & 2\log_{2}\left(\frac{\tau_{\text{max}}}{\tau_{0}}\right)+\frac{25}{2}\tau_{\text{max}}(f(x_{0})-f_{\text{low}})\epsilon^{-2}.
\label{eq:iteration_complexity}
\end{eqnarray}
Since each iteration of Algorithm 1 requires at most $(FE_{g}+2)$ function evaluations and $(RE_{g}+1)$ retractions, it follows from (\ref{eq:iteration_complexity}), that the bounds (\ref{eq:bound1}) and (\ref{eq:bound2}) hold.
\end{proof}

\Cref{theorem0} assumes that $\tau_k$ is bounded for all $k$. We next prove that this condition holds under \Cref{assumption1}. The following lemma shows that the finite-difference gradient defined in \eqref{eq:step11} is an approximate Riemannian gradient in the sense of \Cref{def:approximate_gradient}, with $C_{f}=L_{\mathcal{M}}\sqrt{d}/2$. This allows us to apply \Cref{lemma2} and \Cref{lemma3} to guarantee a sufficiently large norm of the finite-difference gradient when \(x_k\) is not an \(\epsilon\)-critical point of \(f\), as well as a sufficient decrease of the objective function when \(\tau_k\) and \(\sigma_k\) are sufficiently large. As a result, the steps are eventually accepted systematically and \(\tau_k\) does not increase beyond a fixed threshold.

\begin{lemma} \label{lem:intrinsicRFD_is_approx_gradient} Assume that \Cref{assumption1} holds, and let $g_h(x)$ be the intrinsic finite difference gradient defined in \eqref{eq:step11}. Then
\begin{equation*}
\norm{x}{\grad f(x) - g_h(x)} \leq \frac{L_\M \sqrt{d}}{2} h.
\end{equation*}
\end{lemma}

\begin{proof}
Let $\{e_1(x), \dots, e_d(x)\}$ be an orthonormal basis of $T_x \M$, then any tangent vector $\eta \in T_x \M$ can be written as
\[ \eta = \sum_{l=1}^d \langle \eta, e_l(x) \rangle e_l(x). \]
In particular, 
\begin{align*}
    \grad f(x) &= \sum_{l=1}^d \langle \grad f(x), e_l(x) \rangle e_l(x).
\end{align*}
There follows:
\begin{align*}
\norm{x}{g_h(x) - \grad f(x)}^2 &= \norm{x}{\sum_{l=1}^d \frac{f(R_x(h e_l(x)))-f(x)}{h} e_l(x) - \grad f(x)}^2 \\
&= \frac{1}{h^2} \sum_{l=1}^d \left(f(R_x(h e_l(x)))-f(x) - h\inner{x}{\grad f(x)}{e_l(x)} \right)^2 \\
&\leq d \left(\frac{L_\M h}{2} \right)^2,
\end{align*}
where we used the fact that, by \Cref{assumption1}, for $l=1,\dots,d,$ there holds


\begin{equation*}
|f(R_x(h e_l(x))) - f(x) - h \inner{x}{\grad f(x)}{e_l(x)}| \leq \frac{L_\M h^2}{2} \norm{x}{e_l(x)}^2 = \frac{L_\M h^2}{2}. 
\end{equation*}
\end{proof}

The next result provides an upper bound on $\tau_k$ for $k \in \{0, \dots, T(\epsilon)-1\}$.

\begin{lemma} \label{lem:sigma_tau_bounded}
Assume that \Cref{assumption1} holds, and  let $(x_k)_{k=0}^{T(\epsilon)-1}$ be a sequence generated by \Cref{alg:intrinsicRFD}. Then,  for all $k \in \{0,\dots, T(\epsilon)-1\}$, there holds 
\begin{equation}
\tau_k \leq \max\{\tau_0, 4L_\M\} =: \tau_{\max}^{\Int}.
\label{eq:bound3}
\end{equation}
\end{lemma}

\begin{proof} 
We proceed by induction on \(k\). For \(k=0\), \eqref{eq:bound3} holds trivially. 
Assume that \eqref{eq:bound3} holds for some 
\(k \in \{0,\ldots,T(\epsilon)-2\}\). 
We show that \eqref{eq:bound3} also holds for \(k+1\) by considering the following three cases.

\medskip
\noindent\textbf{Case 1:} \(k \in \mathcal{U}^{(1)}\) (i.e., \eqref{eq:step12} holds).

\medskip
\noindent In this case, we must have
\begin{equation}
\tau_k < L_{\mathcal{M}},
\label{eq:bound4}
\end{equation}
since otherwise \Cref{lem:intrinsicRFD_is_approx_gradient} and \Cref{lemma2}, with 
\(C_f = L_{\mathcal{M}}\sqrt{d} / 2\), would imply that \(\|g_{h_k}(x_k)\|_x \geq 4\epsilon/5\), contradicting \eqref{eq:step12}. 
Using \eqref{eq:bound4} and the update rule for \(\tau_k\) in Step~1.2 of \Cref{alg:intrinsicRFD}, we obtain
\[
\tau_{k+1} = 2\tau_k < 2L_{\mathcal{M}} \leq \tau_{\max}^{\mathrm{Int}},
\]
so that \eqref{eq:bound3} holds for \(k+1\).

\medskip
\noindent\textbf{Case 2:} \(k \in \mathcal{U}^{(2)}\) (i.e., \eqref{eq:step13} fails and \(\sigma_{k+1} > \tau_k\)).

\medskip
\noindent In this case, we must have
\begin{equation}
\tau_k < 2L_{\mathcal{M}},
\label{eq:bound5}
\end{equation}
since otherwise
\[
2\sigma_k = \sigma_{k+1} > \tau_k \geq 2L_{\mathcal{M}},
\]
which implies \(\tau_k \geq \sigma_k > L_{\mathcal{M}}\). 
By \Cref{lemma3} and \Cref{lemma2}, with 
\(C_f = L_{\mathcal{M}}\sqrt{d} / 2\), this would imply that \eqref{eq:step13} holds, contradicting the definition of \(\mathcal{U}^{(2)}\).
Therefore, using \eqref{eq:bound5} and the update rule for \(\tau_k\) in Step~1.4 of \Cref{alg:intrinsicRFD}, we obtain
\[
\tau_{k+1} = 2\tau_k < 4L_{\mathcal{M}} \leq \tau_{\max}^{\mathrm{Int}},
\]
and thus \eqref{eq:bound3} holds for \(k+1\).

\medskip
\noindent\textbf{Case 3:} \(k \in \mathcal{U}^{(3)} \cup \mathcal{S}\).

\medskip
\noindent In this case, the update rule for \(\tau_k\) in Step~1.4 and the induction hypothesis imply that
\(\tau_{k+1} = \tau_k \leq \tau_{\max}^{\mathrm{Int}}\).
Hence, \eqref{eq:bound3} also holds in this case.
\end{proof}

Now we can provide a complete complexity result for Algorithm 1.
\begin{corollary} \label{cor:worst_case_complexity_intrinsicRFD}
Under \Cref{assumption1} and \Cref{assumption2}, \Cref{alg:intrinsicRFD} finds an $\epsilon$-critical point of \eqref{eq:P} (i.e., a point $x^*\in \M$ such that $\|\grad f(x^*)\|_{x^*} \leq \epsilon$) with the following worst-case complexity bounds:
\begin{align*}
\text{FE}(\epsilon) &\leq (d+2)\left[2\log_{2}\left(\frac{\max\left\{\tau_{0},4L_{\mathcal{M}}\right\}}{\tau_{0}}\right)+\frac{25}{2}\max\left\{\tau_{0},4L_{\mathcal{M}}\right\}(f(x_{0})-f_{\text{low}})\epsilon^{-2}\right], \label{eq:bound1}\\
\text{RE}(\epsilon) &\leq (d+1)\left[2\log_{2}\left(\frac{\max\left\{\tau_{0},4L_{\mathcal{M}}\right\}}{\tau_{0}}\right)+\frac{25}{2}\max\left\{\tau_{0},4L_{\mathcal{M}}\right\}(f(x_{0})-f_{\text{low}})\epsilon^{-2}\right],
\end{align*}
\end{corollary}

\begin{proof}
It follows directly from \Cref{theorem0} and \Cref{lem:sigma_tau_bounded}, noting that the computation in \eqref{eq:step11} requires \(FE_g = d\) function evaluations in addition to \(f(x_k)\), as well as \(RE_g = d\) retractions.
\end{proof}

Comparing \Cref{cor:worst_case_complexity_intrinsicRFD} with the results from \cite{Grapiglia2023}, we recover the linear dependency in problem dimension (the dimension $d$ of the manifold in our case). Note also that choosing $\tau_0$ larger typically improves the logarithmic terms in the bounds, which is expected since it corresponds to more accurate Riemannian gradient approximations, and equivalently, fewer unsuccessful iterations of type I and III. On the other hand, increasing $\sigma_0$ (and so $\tau_0$) slows down the convergence in general since step size may be too conservative. We therefore recommend to implement the method with $\tau_0$ much larger than $\sigma_0$. Note also that the number of retractions scales similarly to the number of function evaluations, we will show next that it is possible to reduce it drastically when $\M$ is a Riemannian submanifold of $\R^n$.


\section{An extrinsic finite-difference scheme on manifolds}  \label{sec:extrinsic_scheme}

Let us assume now that the objective function $f$ in \eqref{eq:P} is defined over an $n$-dimensional Euclidean space $\E$, but that its minimization is restricted to a $d$-dimensional Riemannian submanifold $\M$ of $\E$. We make the following assumption on $f$.

\begin{assumption} \label{assumption3}
The function $f : \E \to \R$ is $L_\E$-smooth, i.e., for all $x \in \E$,
\begin{equation*}
|f(x+\eta)-f(x) - \langle\nabla f(x),\eta\rangle| \leq \frac{L_\E}{2} \|\eta\|^2, \quad \forall \eta \in \E, 
\end{equation*}
where $\langle \cdot, \cdot \rangle$ and $\| \cdot \|$ are the usual Euclidean inner product and norm, respectively.
\end{assumption}

\begin{remark} \label{rem:linkA1vsA3}
    According to \cite{Boumal2019}, when $\M$ is a compact Riemannian submanifold of $\E$, Assumption \Cref{assumption3} implies Assumption \Cref{assumption1} for some constant $L_\M$ that depends on its Euclidean counterpart, on the diameter of the manifold and on some norm of the differential of the retraction. 
\end{remark}

The goal of this section is to propose an alternative finite-difference scheme, presented in \Cref{alg:extrinsicRFD}, that relies on the embedding of $\M$ in $\E$ to avoid the need of computing retractions in the intrinsic finite-difference scheme used in last section.  This is motivated by the observation that the computational cost of retractions can be substantial in practice. For instance, on the orthogonal group, retractions involve matrix factorizations such as the QR or polar decompositions, whose computation requires a cubic number of floating point operations in terms of the size of the matrix, becoming prohibitive in high dimensions (see, e.g., \cite[Example 4.1.2]{Absil2008}).

\begin{algorithm}[h!]
\caption{Extrinsic Riemannian finite-difference method (Ext-RFD)}
\label{alg:extrinsicRFD}
\textbf{Step 0.} Let $x_0 \in \M$, $\sigma_0 > 0$, $\tau_0 > \sigma_0$, $\epsilon > 0$, and set $k:=0$. Let $d$ be the dimension of $\M$. \\
\textbf{Step 1.1.} Let $h_k = \frac{2\epsilon}{5\sqrt{d}\tau_k}$ and compute the approximate Riemannian gradient 
\begin{equation} \label{eq:finite_diff2}
g_{h_k}(x_k) = \sum_{l=1}^d \frac{f(x_k+h_k e_l(x_k)) - f(x_k)}{h_k} e_l(x_k),
\end{equation}
where $\{e_1(x_k),\dots,e_d(x_k)\}$ is an orthonormal basis of $T_{x_k} \M$. \\
\textbf{Steps 1.2--1.4}: Same as Steps 1.2--1.4 of \Cref{alg:intrinsicRFD}. \\
\end{algorithm}

 
The convergence analysis of \Cref{alg:extrinsicRFD} relies on the following result, which shows that the extrinsic finite-difference scheme \eqref{eq:finite_diff2} leads to an approximate Riemannian gradient in the sense of \Cref{def:approximate_gradient}, with $C_{f}=L_{\E}\sqrt{d}/2$. 

\begin{lemma} \label{lemma_extrinsic_gradient_accuracy}
Let $f : \E \to \R$ and let $\M$ be a $d$-dimensional Riemannian submanifold of $\E$ in \eqref{eq:P}. Assume that \Cref{assumption3} hold, and let $g_h(x)$ be the extrinsic finite-difference gradient of $f$ defined in \eqref{eq:finite_diff2}, then
\begin{equation*}
\norm{x}{\grad f(x) - g_h(x)} \leq \frac{L_\E \sqrt{d}}{2} h,
\end{equation*}
where $\norm{x}{\cdot}$ coincides with the Euclidean norm, since $\M$ is a Riemannian submanifold of $\E$.
\end{lemma}

\begin{proof}
Since $\M$ is a Riemannian submanifold of $\E$, the Riemannian gradient $\grad f(x)$ is the orthogonal projection of its Euclidean counterpart $\nabla f(x)$ onto the tangent space $T_x \M$, i.e., $\grad f(x) = \text{Proj}_{\tangent{x}} \nabla f(x)$ for all $x \in \M$, see \cite[eq. 3.37]{Absil2008}. Let $\{e_1(x), \dots, e_d(x)\}$ be an orthonormal basis of $T_x \M$, then any tangent vector $\eta \in T_x \M$ can be written as
\[ \eta = \sum_{l=1}^d \langle \eta, e_l(x) \rangle e_l(x). \]
Particularizing this equality to the Riemannian gradient gives 
\begin{align*}
    \grad f(x) &= \sum_{l=1}^d \langle \grad f(x), e_l(x) \rangle e_l(x) = \sum_{l=1}^d \langle \text{Proj}_{\tangent{x}} \nabla f(x), e_l(x) \rangle e_l(x) = \sum_{l=1}^d \langle \nabla f(x), e_l(x) \rangle e_l(x),
\end{align*}
by definition of the orthogonal projection. There follows:
\begin{align*}
\|g_h(x) - \grad f(x)\|^2
&= \left\| \sum_{l=1}^d \left(\frac{f(x+h e_l(x))-f(x)}{h} - \langle \nabla f(x), e_l(x) \rangle \right) e_l(x) \right\|^2 \\
&= \frac{1}{h^2} \sum_{l=1}^d \left(f(x+h e_l(x))-f(x) - h \langle \nabla f(x), e_l(x) \rangle \right)^2 \\
&\leq d \left( \frac{L_\E h}{2} \right)^2,
\end{align*}
where the inequality results from Assumption \Cref{assumption3}, which implies that, for $l=1,\dots,d$
\begin{equation*}
|f(x+h e_l(x)) - f(x) - h \langle \nabla f(x), e_l(x)\rangle| \leq \frac{L_\E h^2}{2} \|e_l(x)\|^2 = \frac{L_\E h^2}{2}. 
\end{equation*}
\end{proof}

Similarly as in the previous section, we can use this upper bound on the relative error on gradient approximations to derive upper bounds on the  parameters $\sigma_k$ and $\tau_k$ in \Cref{alg:extrinsicRFD}, for $k=0,\dots,T( \epsilon)-1$. For this, we rely on the four sets of iterations $ \mathcal{U}^{(1)}, \mathcal{U}^{(2)}, \mathcal{U}^{(3)},\mathcal{S}$, defined in \Cref{sec:mainDRFOmethod} for \Cref{alg:intrinsicRFD}, and whose definition can be identically transferred to \Cref{alg:extrinsicRFD}. 


\begin{lemma} \label{lem:sigma_tau_bounded_extrinsic}
Let $f : \E \to \R$ and let $\M$ be a $d$-dimensional Riemannian submanifold of $\E$. Assume that \Cref{assumption1} and \Cref{assumption3} hold and let $(x_k)_{k=0}^{T(\epsilon)-1}$ be a sequence generated by \Cref{alg:extrinsicRFD}. Then, for all $k \in \{0,\dots, T(\epsilon)-1\}$, 
\begin{equation*}
\tau_k \leq \max\{\tau_0, 4L_\M, 4L_\E\} =: \tau_{\max}^{\Ext}.
\end{equation*}
\end{lemma}

\begin{proof}
It follows exactly as in the proof of \Cref{lem:sigma_tau_bounded}, by using \Cref{lemma_extrinsic_gradient_accuracy} with $C_{f}=L_{\E}\sqrt{d}/2$. 
\end{proof}

\begin{remark}
    Note that, while the  sufficient decrease of the objective is  only governed by the Lipschitz constant of the gradient along the manifold (written $L_\M$, see \Cref{lemma3}), the minimum accuracy of the extrinsic finite-difference gradient required in \Cref{lemma_extrinsic_gradient_accuracy} and \Cref{lemma2} depends on the Lipschitz constant $L_\E$ of the gradient of the objective in the Euclidean space, which will also account for variations of the objective in directions normal to the manifold. In other words, while by \Cref{rem:linkA1vsA3}, Lipschitz-smoothness of $f$ in $\E$ ensures Lipschitz-smoothness of $f$ on the manifold when the latter is compact, $L_\E$ may be much larger than $L_\M$. This is accounted for in \Cref{alg:extrinsicRFD} by the fact that repeated failures of the descent condition in Step 1.3 will lead to a progressive increase of $\sigma_k$, possibly going beyond $L_\M$, triggering in Step 1.4 an associated increase of $\tau_k$, until the extrinsic finite-difference scheme gets sufficiently accurate to allow for a decrease of the objective in Step 1.3, which will occur as soon as $\tau_k \geq L_\E$.
\end{remark}

Note that Theorem \ref{theorem0} also applies to \Cref{alg:extrinsicRFD}. Since in \Cref{alg:extrinsicRFD}, each computation of the finite-difference gradient requires \(FE_g = d\) function evaluations in addition to \(f(x_k)\), and \(RE_g = 0\) retractions, we obtain the following complexity result.


\begin{corollary} \label{corollary2}
Let $f : \E \to \R$ and let $\M$ be a $d$-dimensional Riemannian submanifold of $\E$. Assume that \Cref{assumption1}, \Cref{assumption2} and \Cref{assumption3} hold. Then, \Cref{alg:extrinsicRFD} finds an $\epsilon$-critical point of \eqref{eq:P} (i.e., a point $x^*\in \M$ such that $\|\grad f(x^*)\|_{x^*} \leq \epsilon$) with the following
worst-case complexity bounds:
\begin{align*}
\text{FE}(\epsilon) &\leq (d+2)\left[2\log_{2}\left(\frac{\max\left\{\tau_{0},4L_{\E},4L_{\mathcal{M}}\right\}}{\tau_{0}}\right)+\frac{25}{2}\max\left\{\tau_{0},4L_{\E},4L_{\mathcal{M}}\right\}(f(x_{0})-f_{\text{low}})\epsilon^{-2}\right], \label{eq:bound1}\\
\text{RE}(\epsilon) &\leq 2\log_{2}\left(\frac{\max\left\{\tau_{0},4L_{\E},4L_{\mathcal{M}}\right\}}{\tau_{0}}\right)+\frac{25}{2}\max\left\{\tau_{0},4L_{\E},4L_{\mathcal{M}}\right\}(f(x_{0})-f_{\text{low}})\epsilon^{-2}.
\end{align*}
\end{corollary}

\begin{proof}
It follows directly from \Cref{theorem0} combined with \Cref{lem:sigma_tau_bounded_extrinsic}.
\end{proof}


Note that \Cref{corollary2} implies that $\text{RE}(\epsilon)$ does not depend on the problem dimension, because retractions are no longer used in the extrinsic finite-difference scheme. This feature makes the method particularly well-suited for large-scale problems where the retractions are expensive to compute. On the other hand, this may come at the cost of an increased number of function evaluations. Indeed, the smoothness constant $L_\E$ may be larger than $L_\M$, which means that more function evaluations may be needed to build sufficiently accurate finite-difference gradients. The worst-case complexity rates of the two methods proposed in this paper are summarized in \Cref{table:2}.

\begin{table}[H]
\centering
\begin{tabular}{|c|c|c|}
\hline
& Int-RFD & Ext-RFD \\
\hline
$\text{FE}(\epsilon)$ & $O(L_\M d\epsilon^{-2})$ & $O(\max\{L_\M, L_\E\} d \epsilon^{-2})$ \\
\hline
$\text{RE}(\epsilon)$ & $O(L_\M d\epsilon^{-2})$ & $ O(\max\{L_\M, L_\E\} \epsilon^{-2})$ \\
\hline
\end{tabular}
\caption{Worst-case complexity bounds to reach an $\epsilon$-critical point in terms of number of function evaluations $\text{FE}(\epsilon)$ and number of retractions $\text{RE}(\epsilon)$.}
\label{table:2}
\end{table}

\section{Numerical results} \label{sec:numerics}

We compare here numerically our proposed algorithms to existing derivative-free methods. The first part of the section addresses the Euclidean setting, where we compare \Cref{alg:intrinsicRFD} to the finite-difference method proposed in \cite{Grapiglia2023}. We then compare \Cref{alg:intrinsicRFD} and \Cref{alg:extrinsicRFD} to the DFRO method proposed  in \cite{Kungurtsev2023} on three common Riemannian optimization problems. The implementation of \Cref{alg:intrinsicRFD} and \Cref{alg:extrinsicRFD} are available on \url{https://github.com/taminiaut/Extrinsic-Intrinsic-Riemannian-Finite-Difference-Method}.

\subsection{Experiments in the Euclidean setting}

Our proposed algorithms (\Cref{alg:intrinsicRFD} and \Cref{alg:extrinsicRFD}) rely on two sequences of parameters $(\sigma_k)_{k\geq 0}$, $(\tau_k)_{k\geq 0}$ to account for the smoothness of the objective in the stepsize and finite-difference accuracy. In contrast the Derivative-Free Quadratic Regularization Method (DFQRM), proposed in \cite{Grapiglia2023}, only uses one such sequence of parameters (equivalently, $\sigma_k$ is constrained to be equal to $\tau_k$ at each iteration). In this section, we show numerically that  relying on two sequences results in a faster convergence compared to DFQRM.  Our comparison is run on the OPM test set proposed in \cite{Gratton2021} which is a subset of 134 problems the well-known CUTEst collection \cite{Gould2015}. Note also that, since, in the Euclidean setting, $\M = \R^n$ in \eqref{eq:P}, our proposed Int-RFD and Ext-RFD algorithms (respectively, \Cref{alg:intrinsicRFD} and \Cref{alg:extrinsicRFD}) are equivalent, hence we only consider here the Int-RFD algorithm. Note that, for the sake of comparison, we do not use Hessian approximations in DFQRM, and choose exactly the same parameters for DFQRM (in particular, $\sigma_0=1$ and $\epsilon=10^{-5}$) as for \Cref{alg:intrinsicRFD} (with $\tau_{0}=100$).

We use the performance profiles introduced in \cite{Dolan2002} to compare the two methods. Let $\mathcal{P}$ be a set of problems, $\mathcal{S}$ a set of solvers and $\eta \in (0,1)$ a tolerance, we measure the computational cost $t_{p,s}$ (either in terms of number of function evaluations, or running time) required by any solver $s \in \mathcal{S}$ to find a point $x$ that satisfies the following convergence test, on problem $p \in \mathcal{P}$ and for some initial iterate $x_0$:
\begin{equation} \label{eq:convergence_test}
f(x_0) - f(x) \geq (1-\eta) (f(x_0) - f_\text{best}),
\end{equation}
where $f_\text{best}$ is the best function value found by any of the solver in $\mathcal{S}$ within a fixed maximum budget of $100(d+1)$ function evaluations (where $d$ is the dimension of the manifold). 
The performance profile curve of the solver $s$ is then given by
\begin{equation*}
\rho_s(\alpha) = \frac{1}{|\mathcal{P}|} \left|\left\{ p \in \mathcal{P} \, : \, \frac{t_{p,s}}{\min_{s' \in \mathcal{S}} t_{p,s'}} \leq \alpha \right\}\right|.
\end{equation*}
Performance profiles allow to compare simultaneously the efficiency and robustness of the methods. The value $\rho_s(1)$ indicates the proportion of problems that were solved, according to \eqref{eq:convergence_test}, by the method $s$ with the smallest budget; more efficient methods are associated with a larger $\rho_s(1)$. On the other hand, more robust methods, i.e., methods that can solve more problems possibly with an increased budget, will have higher $\rho_s(\alpha)$ for $\alpha$ large. In our experiments, we chose $\eta = 10^{-3}$ in \eqref{eq:convergence_test} and we relied on the implementation of performance profiles from \cite{More2009}. The resulting performance profiles in terms of function evaluations and running time are presented in \Cref{fig:0a} and \Cref{fig:0b}, respectively. These plots show that using two distinct estimates makes the method better on almost all the problems and it is also better in terms of robustness. 

\begin{figure}[h]
\centering
\begin{subfigure}[b]{0.5\textwidth}
\includegraphics[width=\textwidth]{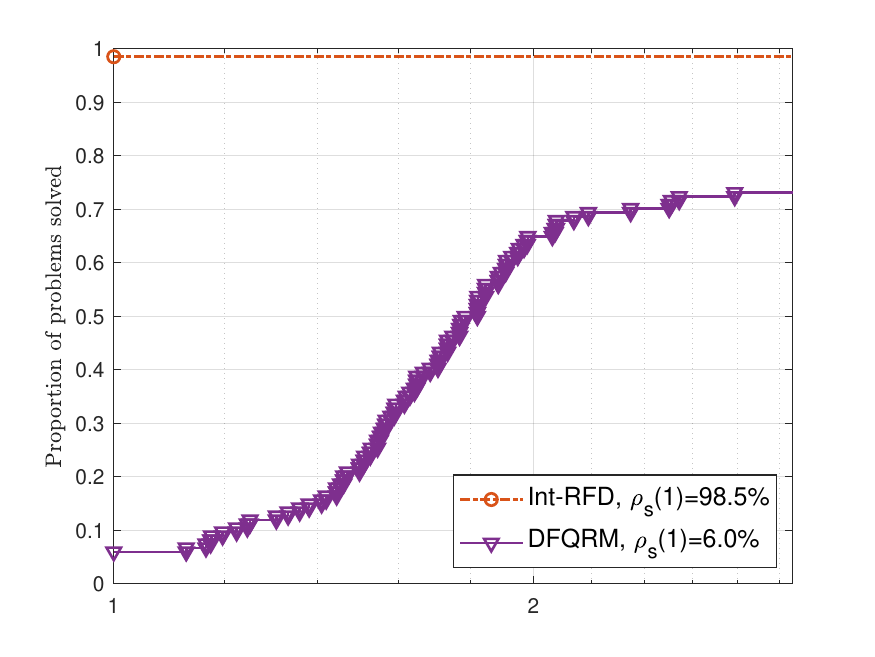}
\caption{Performance profiles (function evaluations).}
\label{fig:0a}
\end{subfigure}%
~ 
\begin{subfigure}[b]{0.5\textwidth}
\includegraphics[width=\textwidth]{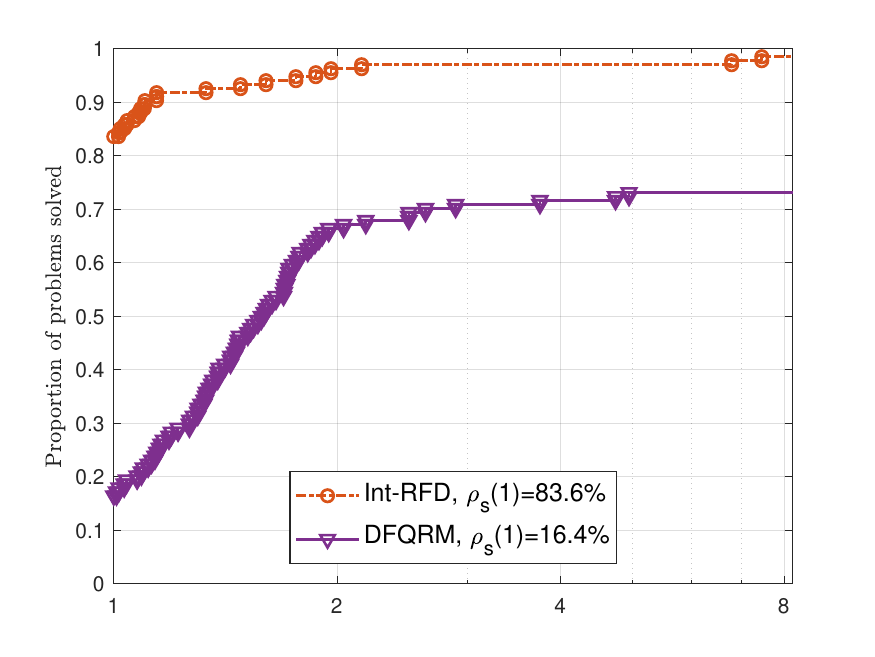}
\caption{Performance profiles (running time).}
\label{fig:0b}
\end{subfigure}
\caption{Numerical experiments in the Euclidean case on the OPM test set \cite{Gratton2021}.}
\end{figure}


\subsection{Experiments on Riemannian test sets}
Secondly, we compare our methods with the RDSE-SB algorithm proposed in \cite{Kungurtsev2023}, on three smooth Riemannian optimization problems. We constructed our test set based on these problem implementations from \cite{Kungurtsev2023}, considering for each problem several random instances and several problem dimensions.

\Cref{alg:intrinsicRFD} and \Cref{alg:extrinsicRFD} are implemented with $\sigma_0 = 1$, $\tau_0 = 100$ and $\epsilon = 10^{-5}$. As all problems considered are defined on Riemannian submanifolds, orthonormal bases of the tangent space were generated by projecting random vectors of the ambient space into the tangent space, then applying the Gram-Schmidt procedure to these vectors to make them orthonormal. 

\noindent For the RDSE-SB method we used the implementations provided in \cite{Kungurtsev2023}, while the manifold structures are handled thanks to the MANOPT toolbox \cite{Boumal2014}. We next describe the three test set problems selected from \cite{Kungurtsev2023} that we used in this work. For each problem, we generate a set of random problem instances of different dimensions given in \Cref{table:1}. To give the reader a better view on the associated manifold and ambient space dimensions (respectively $d$ and $n$), we also include their values in \Cref{table:1}.

\paragraph{Problem I : Top singular vectors.}
The problem of finding the $m_3$ leading left and right singular vectors (associated to the largest singular values) of a matrix $A \in \R^{m_1 \times m_2}$ is formulated as the Riemannian optimization problem
\begin{equation*}
\min_{\substack{X \in St(m_1,m_3) \\ Y \in St(m_2,m_3)}} -\text{trace}(X^T A Y),
\end{equation*}
where $St(m,p) = \{ X \in \R^{m \times p} \; | \; X^T X = I \}$ for $m \geq p$, is the Stiefel manifold, seen as a Riemannian submanifold of $\R^{m \times p}$. 

\paragraph{Problem II : Dictionary learning.}

Let $Y \in \R^{m_1 \times m_2}$ be a matrix whose columns are $m_1$-dimensional vectors that we would like to represent as well as possible by a sparse linear combination of the columns of a dictionary $D \in \R^{m_1 \times m_3}$, that is, we search a dictionary $D \in \R^{m_1 \times m_3}$ and a (sparse) matrix of coefficients $C \in \R^{m_3 \times m_2}$ such that
\begin{equation*}
\min_{\substack{D \in Ob(m_1,m_3)\\C \in \R^{m_3 \times m_2}}} \|Y - D C\|_F^2 + \lambda \phi(C)
\end{equation*}
where $Ob(m,p) = \{ X = (x_1 \dots x_p) \in \R^{m \times p} \; | \; \|x_i\|_2 = 1, \; \text{for} \; i = 1,\dots, p \}$ is the oblique manifold, seen as a Riemannian submanifold of $\R^{m \times p}$, and $\phi(C) = \sum_{i=1}^p \sum_{j=1}^n \sqrt{C_{ij}^2 + \delta^2}$ is a smoothed $l_1$ norm that promotes sparsity. Note that the oblique manifold constraint allows getting rid of the invariance under scaling of the columns of $D$ and rows of $C$. In our experiments, we set the regularization and smoothing parameters to $\lambda = 0.01$ and $\delta = 0.001$, respectively.

\paragraph{Problem III : Rotation synchronization.}
The goal here is to estimate $m_2$ rotation matrices $R_1, \dots, R_{m_2} \in \R^{m_1 \times m_1}$ from noisy measurements $H_{ij} \approx R_i R_j^{-1}$ for all $i,j = 1, \dots, m_2$, leading to the optimization problem
\begin{equation*}
\min_{R_1, \dots, R_{m_2} \in SO(m_1)} \sum_{i=1}^{m_2} \sum_{j=1}^{i-1} \|R_i - H_{ij} R_j \|_F^2,
\end{equation*}
where $SO(m) = \{ X \in \R^{m \times m} \; | \; X^T X = I, \det(X) = 1 \}$ is the special orthogonal group, seen as a Riemannian submanifold of $\R^{m \times m}$.

\begin{table}[h!]
\centering
\begin{tabular}{|ccccc|c|ccccc|c|ccccc|}
\cline{1-5} \cline{7-11} \cline{13-17}
\multicolumn{5}{|c|}{Problem I} & & \multicolumn{5}{c|}{Problem II} & & \multicolumn{5}{c|}{Problem III} \\
\multicolumn{5}{|c|}{Top singular vectors} & & \multicolumn{5}{c|}{Dictionary learning} & & \multicolumn{5}{c|}{Rotations synchronization} \\
\cline{1-5} \cline{7-11} \cline{13-17}
\multicolumn{1}{|c|}{$m_1$} & \multicolumn{1}{c|}{$m_2$} & \multicolumn{1}{c|}{$m_3$} & \multicolumn{1}{c|}{$n$} & \multicolumn{1}{c|}{$d$} & &
\multicolumn{1}{|c|}{$m_1$} & \multicolumn{1}{c|}{$m_2$} & \multicolumn{1}{c|}{$m_3$} & \multicolumn{1}{c|}{$n$} & \multicolumn{1}{c|}{$d$} & &
\multicolumn{1}{|c|}{$m_1$} & \multicolumn{1}{c|}{$m_2$} & \multicolumn{1}{c|}{$m_3$} & \multicolumn{1}{c|}{$n$} & \multicolumn{1}{c|}{$d$} \\ 
\cline{1-5} \cline{7-11} \cline{13-17}
2 & 2 & \multicolumn{1}{c|}{1} & 4 & 2 &&
2 & 3 & \multicolumn{1}{c|}{1} & 5 & 4 &&
2 & 2 & \multicolumn{1}{c|}{$\bigtimes$} & 8 & 2  \\
3 & 3 & \multicolumn{1}{c|}{2} & 12 & 16 &&
3 & 5 & \multicolumn{1}{c|}{2} & 16 & 14 &&
2 & 4 & \multicolumn{1}{c|}{$\bigtimes$} & 16 & 4 \\
5 & 5 & \multicolumn{1}{c|}{2} & 20 & 14 &&
4 & 6 & \multicolumn{1}{c|}{3} & 30 & 27 &&
2 & 6 & \multicolumn{1}{c|}{$\bigtimes$} & 24 & 6 \\
10 & 10 & \multicolumn{1}{c|}{2} & 40 & 34 &&
5 & 7 & \multicolumn{1}{c|}{4} & 48 & 44 &&
4 & 2 & \multicolumn{1}{c|}{$\bigtimes$} & 32 & 12 \\
15 & 15 & \multicolumn{1}{c|}{2} & 60 & 54 &&
6 & 8 & \multicolumn{1}{c|}{5} & 70 & 65 &&
4 & 4 & \multicolumn{1}{c|}{$\bigtimes$} & 64 & 24 \\
20 & 20 & \multicolumn{1}{c|}{2} & 80 & 74 &&
8 & 10 & \multicolumn{1}{c|}{6} & 108 & 102 &&
4 & 6 & \multicolumn{1}{c|}{$\bigtimes$} & 96 & 36 \\
30 & 30 & \multicolumn{1}{c|}{2} & 120 & 114 &&
10 & 12 & \multicolumn{1}{c|}{7} & 154 & 147 &&
6 & 2 & \multicolumn{1}{c|}{$\bigtimes$} & 72 & 30 \\
5 & 5 & \multicolumn{1}{c|}{4} & 40 & 20 &&
12 & 14 & \multicolumn{1}{c|}{8} & 208 & 200 &&
6 & 4 & \multicolumn{1}{c|}{$\bigtimes$} & 144 & 60 \\
10 & 10 & \multicolumn{1}{c|}{4} & 80 & 60 &&
14 & 16 & \multicolumn{1}{c|}{10} & 300 & 290 &&
6 & 6 & \multicolumn{1}{c|}{$\bigtimes$} & 216 & 90 \\
20 & 20 & \multicolumn{1}{c|}{4} & 160 & 140 &&
5 & 20 & \multicolumn{1}{c|}{3} & 75 & 72 &&
8 & 2 & \multicolumn{1}{c|}{$\bigtimes$} & 128 & 56 \\
30 & 30 & \multicolumn{1}{c|}{4} & 240 & 220 &&
7 & 20 & \multicolumn{1}{c|}{5} & 135 & 130 &&
8 & 4 & \multicolumn{1}{c|}{$\bigtimes$} & 256 & 112 \\
30 & 10 & \multicolumn{1}{c|}{6} & 240 & 198 &&
12 & 20 & \multicolumn{1}{c|}{3} & 96 & 93 &&
8 & 6 & \multicolumn{1}{c|}{$\bigtimes$} & 384 & 168 \\
30 & 15 & \multicolumn{1}{c|}{6} & 270 & 228 &&
3 & 20 & \multicolumn{1}{c|}{5} & 115 & 110 &&
10 & 2 & \multicolumn{1}{c|}{$\bigtimes$} & 200 & 90 \\
30 & 10 & \multicolumn{1}{c|}{8} & 320 & 248 &&
5 & 20 & \multicolumn{1}{c|}{7} & 175 & 168 &&
10 & 4 & \multicolumn{1}{c|}{$\bigtimes$} & 400 & 180 \\
30 & 15 & \multicolumn{1}{c|}{8} & 360 & 288 &&
3 & 20 & \multicolumn{1}{c|}{12} & 276 & 264 &&
10 & 6 & \multicolumn{1}{c|}{$\bigtimes$} & 600 & 270 \\
\cline{1-5} \cline{7-11} \cline{13-17}
\end{tabular}
\caption{Values of the parameters $(m_1,m_2,m_3)$ of the problem instances included in our test set, and corresponding dimensions of the ambient space ($n$) and of the manifold ($d$).}
\label{table:1}
\end{table}

\begin{figure}[H]
\centering
\begin{subfigure}[b]{0.5\textwidth}
\includegraphics[width=\textwidth]{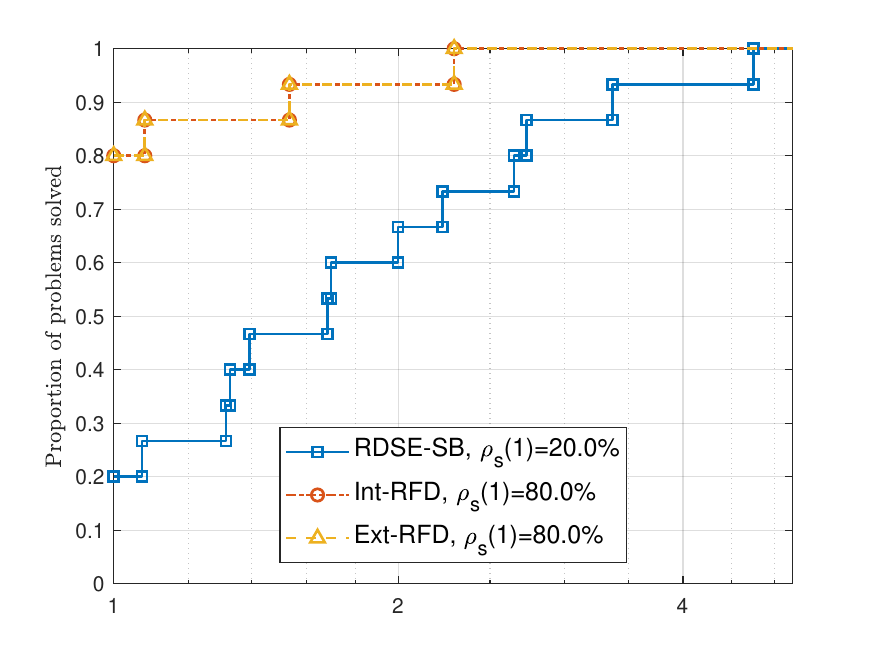}
\end{subfigure}%
~ 
\begin{subfigure}[b]{0.5\textwidth}
\includegraphics[width=\textwidth]{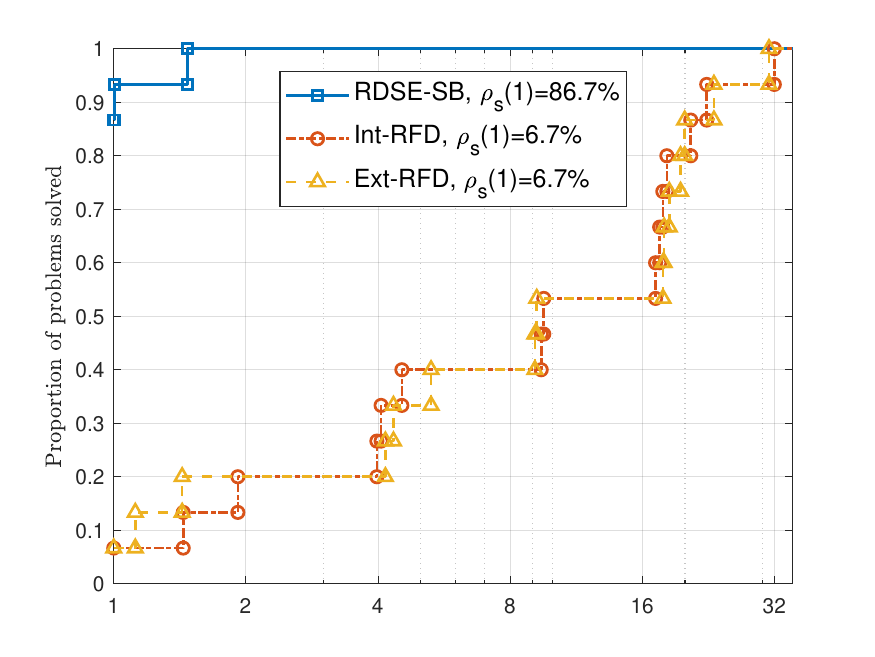}
\end{subfigure}
~
\begin{subfigure}[b]{0.5\textwidth}
\includegraphics[width=\textwidth]{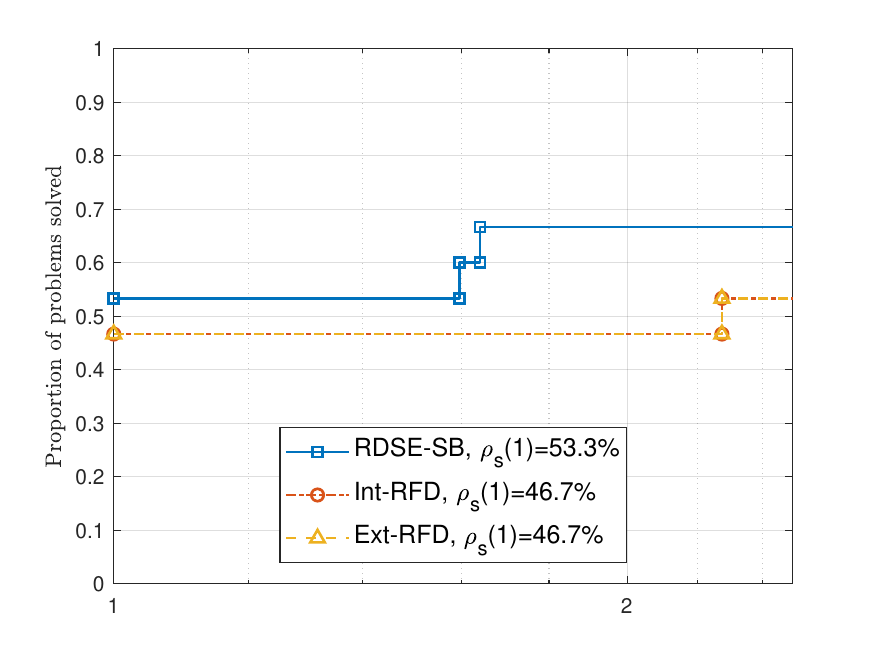}
\end{subfigure}%
~ 
\begin{subfigure}[b]{0.5\textwidth}
\includegraphics[width=\textwidth]{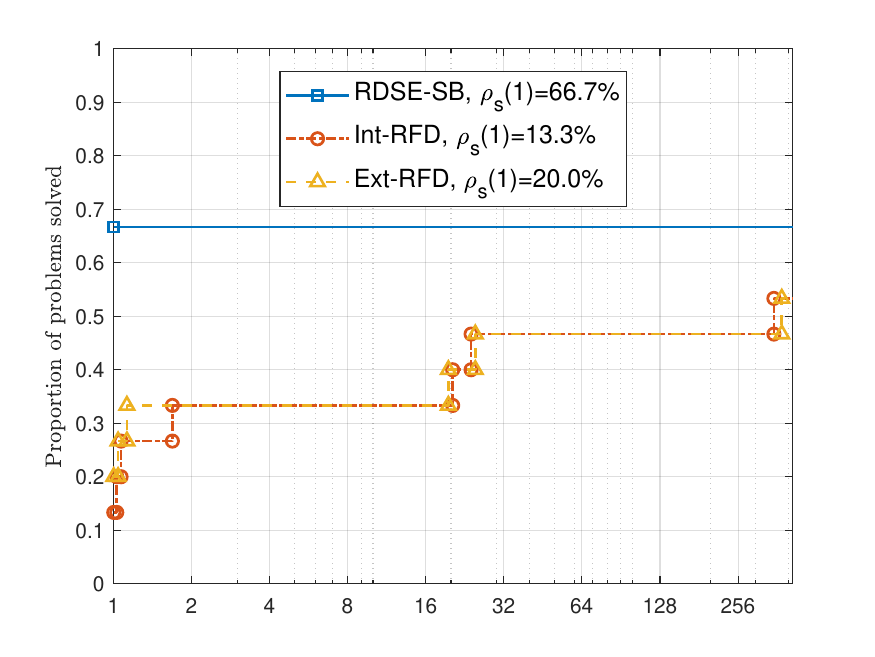}
\end{subfigure}
~
\begin{subfigure}[b]{0.5\textwidth}
\includegraphics[width=\textwidth]{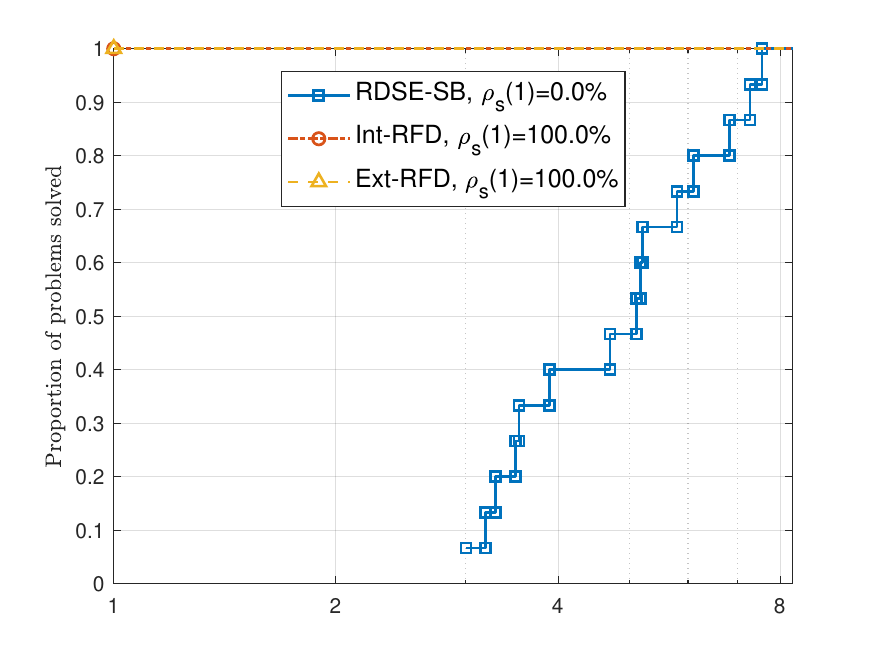}
\caption{Performance profiles (function evaluations).}
\label{fig:3a}
\end{subfigure}%
~ 
\begin{subfigure}[b]{0.5\textwidth}
\includegraphics[width=\textwidth]{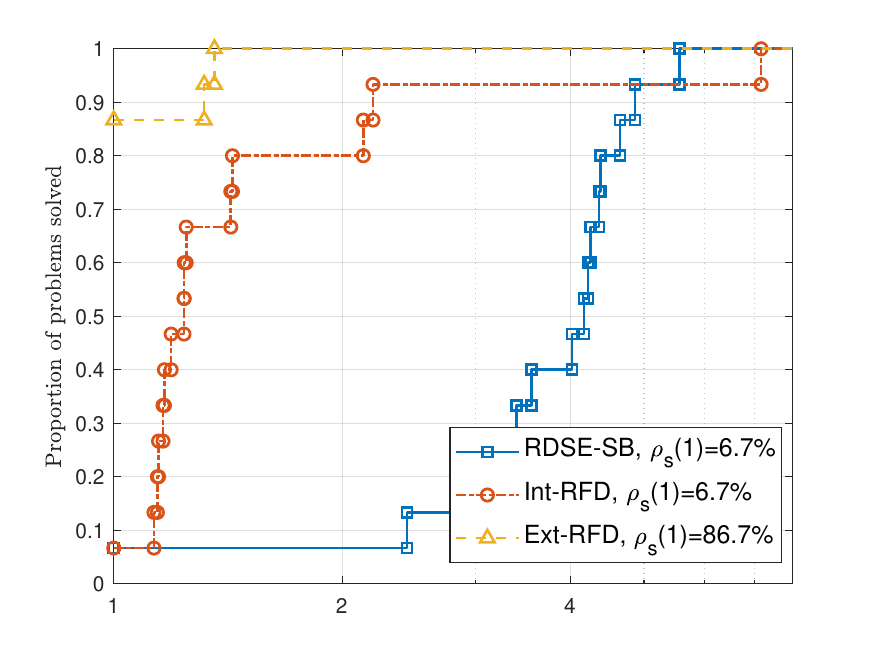}
\caption{Performance profiles (running time).}
\label{fig:3b}
\end{subfigure}
\caption{Experiments for the top singular vectors problem (top row), the dictionary learning problem (second row) and the rotation synchronization problem (bottom row), on a set of problem instances described in \Cref{table:1}.}
\label{fig:numerics}
\end{figure}

As can be seen in \Cref{fig:3a}, the Int-RFD and Ext-RFD methods strongly outperform the RSDE-SB method of \cite{Kungurtsev2023} in two of the three problems regarding accuracy vs function evaluations. For the dictionary learning problem, higher efficiency and robustness were achieved by the RSDE-SB method, though no method manages to solve all problem instances.  We also notice that the Ext-RFD methods remains robust with respect to the Int-RFD method, despite having a worse complexity bound for function evaluations. Considering the additional computational cost of handling the manifold structure, we see in \Cref{fig:3b} that the comparison between Int-RFD and Ext-RFD in terms of running time depends on the manifold considered. When the manifold dimension is much  smaller than that of the ambient space ($d \ll n$), substantial benefits can be achieved by the Ext-RFD method, see the performance profile of the rotation synchronization problem. The comparable running times of Int-RDF and Ext-RFD on other problems is due to the fact that  the Ext-RFD method still requires an orthonormal basis of the tangent space, which is computed here numerically using Gram-Schmidt; the computational cost of the later dominates the cost of retractions on some manifolds. Although closed-form expressions for orthonormal bases of the tangent space exist for all manifolds considered in this section, we include their computation in the cost to present numerical results that reflect performance in generic Riemannian optimization problems.

\section{Conclusions} \label{sec:5}

In this paper, we introduced a novel Riemannian finite-difference method that generalizes the (Euclidean) derivative-free method from \cite{Grapiglia2023}, relying on  two sequences approximating the smoothness constant of the problem in an optimistic and conservative way, respectively.  This strategy, which was inspired from \cite{Davar2025}, amounts to use different smoothness constant estimates for the stepsize than for the finite-difference gradient accuracy. The resulting algorithm, which is novel even in the Euclidean setting, was showed numerically to converge faster that the original method proposed in \cite{Grapiglia2023}. The second contribution of this work is to extend this finite-difference algorithm to the Riemannian setting. For this, we relied on the usual definition of finite differences on Riemannian manifolds, leading to our Int-RFD algorithm presented in \Cref{alg:intrinsicRFD}. We derive worst-case complexity bounds for Int-RFD, showing that it reaches an $\epsilon$-critical point after at most  $\mathcal{O}(d \epsilon^{-2})$ function evaluations and retractions. Since retractions are often computationally intensive, we propose a variant of our algorithm, which we call Ext-RFD (see \Cref{alg:extrinsicRFD}), that relies on an extrinsic finite-difference scheme, for optimization over Riemannian submanifolds of a Euclidean space $\E$. We prove that this second method reaches an $\epsilon$-critical point after at most  $\mathcal{O}(d \epsilon^{-2})$ function evaluations, while  the number of retractions is now independent of problem dimension. The numerical experiments indicate an improvement over the derivative-free method proposed in \cite{Grapiglia2023} in the Euclidean setting and that the extrinsic method significantly outperforms the intrinsic approach when the manifold dimension is much smaller than the dimension of the embedding space. Nevertheless, one should be careful that the worst-case complexity of the low cost method involves the Euclidean smoothness constant $L_\E$ which can be substantially larger than its Riemannian counterpart when the objective varies strongly in directions normal to the manifold in the embedding space.

\bibliographystyle{plain}  
\bibliography{preprint}

\end{document}